\documentclass[11pt,reqno]{amsart}

\usepackage{amsmath, amsthm, amssymb}
\usepackage{amsfonts}
\usepackage[ansinew]{inputenc}
\usepackage[dvips]{epsfig}
\usepackage{graphicx}
\usepackage[english]{babel}
\usepackage{comment}
\usepackage{graphicx}
\usepackage{hyperref}
\usepackage{bbm} 

\usepackage{cite}
\usepackage{graphicx}
\usepackage{amscd}
\usepackage{xcolor}
\usepackage{bm}
\usepackage{enumerate}

\usepackage{verbatim}
\usepackage{hyperref}
\usepackage{amstext}
\usepackage{latexsym}
\let\oldsqrt\sqrt
\def\sqrt{\mathpalette\DHLhksqrt}
\def\DHLhksqrt#1#2{%
\setbox0=\hbox{$#1\oldsqrt{#2\,}$}\dimen0=\ht0
\advance\dimen0-0.2\ht0
\setbox2=\hbox{\vrule height\ht0 depth -\dimen0}%
{\box0\lower0.4pt\box2}}

\allowdisplaybreaks

\newcommand{\R}{\mathbb{R}} 

\newcommand{\ov}{\overline}

\renewcommand{\phi}{\varphi}

\newcommand{\cF}{{\mathcal F}}

\newcommand{\cH}{{\mathcal H}}

\newcommand{\cL}{{\mathcal L}}

\newcommand{\cN}{{\mathcal N}}

\newcommand{\cP}{{\mathcal P}}
\newcommand{\cQ}{{\mathcal Q}}

\theoremstyle{definition}
\newtheorem{defi}{Definition}[section]
\newtheorem{remark}[defi]{Remark}

\theoremstyle{plain} 
\newtheorem{thm}[defi]{Theorem}

\newtheorem{lemma}[defi]{Lemma}

\theoremstyle{definition}

\numberwithin{equation}{section}

 \title[Mixed local and nonlocal supercritical Neumann problem]{Radial positive solutions for mixed local and nonlocal supercritical Neumann problem} 

\author[David Amundsen, Abbas Moameni and Remi Yvant Temgoua]{David Amundsen$^1$, Abbas Moameni$^1$ and Remi Yvant Temgoua$^1$}

\address{$^1$ School of Mathematics and Statistics, Carleton University, Ottawa, Ontario, Canada.}

\email{dave@math.carleton.ca} 

\email{momeni@math.carleton.ca}

\email{remiyvanttemgoua@cunet.carleton.ca}

\date{\today}

\begin{document}

\begin{abstract}
	In this paper, we establish the existence of positive non-decreasing radial solutions for a nonlinear mixed local and nonlocal Neumann problem in the ball. No growth assumption on the nonlinearity is required.  We also provide a criterion for the existence of non-constant solutions provided the problem possesses a trivial constant solution.
\end{abstract}

\maketitle

{\footnotesize
	\begin{center}
	
	\end{center}
}

\section{Introduction and main results}\label{section:introduction}

Given $s\in(\frac{1}{2},1)$, $2\leq p, q$ with $p\neq q$, we consider the following mixed local and nonlocal Neumann problem
 
\begin{equation}\label{e2}
     \left\{\begin{aligned}
     \cL u+u&=a(|x|)|u|^{p-2}u-b(|x|)|u|^{q-2}u~~~\text{in}~~B_1 \\
     \cN_su&=0 \quad\quad\quad\quad\quad\quad\quad\quad\quad\quad\quad\quad\quad\text{in}~~\R^N\setminus\overline{B}_1\\
     \frac{\partial u}{\partial\nu}&=0 \quad\quad\quad\quad\quad\quad\quad\quad\quad\quad\quad\quad\quad\text{on}~~\partial B_1,
     \end{aligned}
     \right.
	\end{equation}
where $B_1$ is the unit ball in $\R^N,~N\geq2$. Under mild assumptions on $a$ and $b$, we study the existence of radial solutions of \eqref{e2} without any restriction on the growths $p$ and $q$. Here, $a$ and $b$ are two non-negative radial functions with some monotonicity assumptions, and $\cL:=-\Delta +(-\Delta)^s$ is the so-called mixed local and nonlocal operator. Recall that $-\Delta$ denotes the classical Laplace operator and $(-\Delta)^s$ the standard fractional Laplacian defined for every sufficiently regular function $u: \R^N\rightarrow\R$ by 
\begin{equation*}
(-\Delta)^su(x)=c_{N,s}P.V.\int_{\R^N}\frac{u(x)-u(y)}{|x-y|^{N+2s}}\ dy,~~~x\in\R^N,
\end{equation*}
where $c_{N,s}$ is a normalization constant and ``$P.V$'' stands for the Cauchy principle value. Moreover, $\nu$ is the exterior normal to $B_1$ and $\cN_s$ is the nonlocal Neumann condition introduced in \cite{dipierro2017nonlocal} defined by
\begin{equation}
    \cN_su(x)=\int_{B_1}\frac{u(x)-u(y)}{|x-y|^{N+2s}}\ dy\quad\quad\text{for all}~~x\in\R^N\setminus\overline{B}_1.
\end{equation}
Neumann problems involving mixed local and nonlocal operators are explored to a very limited extent in the literature. In this regard, we can refer to  \cite{dipierro2022linear,dipierro2022non,mugnai2022mixed}. In \cite{dipierro2022linear}, the authors discussed the spectral properties of a weighted eigenvalue problem and presented a global bound for subsolutions. In \cite{dipierro2022non} a logistic equation is investigated. Very recently, Mugnai and Lippi \cite{mugnai2022mixed} obtained an existence result for a nonlinear problem governed by a mixed operator of type $\cL$ with Neumann conditions. In this paper, the nonlinearity is of a subcritical type.

The aim of this paper is to investigate nonlinear problems for $\cL=-\Delta+(-\Delta)^s$ with Neumann conditions. The nonlinearity that we consider has no growth condition, allowing for supercritical growth. To the best of our knowledge, such types of problems have not yet been treated in the literature. Our motivation comes from the study of a nonlinear problem of the form
\begin{equation}\label{sample}
    \cP u+u=\cF(u)~~\text{in}~~\Omega,
\end{equation}
with some Neumann conditions on $u$. Here, $\Omega$ is a domain in $\R^N$, $\cP$ is some operator (local/nonlocal/mixed) and $\cF$ is a nonlinear function with supercritical growth. One of the main questions addressed in \eqref{sample} is the existence of solutions. We briefly present a literature survey of problems of the type \eqref{sample}.

\begin{itemize}
    \item Local case, i.e., $\cP=-\Delta$: Serra and Tilli considered in \cite{serra2011monotonicity} the Neumann problem
    \begin{equation}\label{classical-neumann-problem}
        -\Delta u+u= a(|x|)f(u),~~~~~u>0~~~~~\text{in}~~B,\quad\quad\frac{\partial u}{\partial\nu}=0~~\text{on}~~\partial B
    \end{equation}
    and they have used a monotonicity constraint to show that, under some assumptions on $a$ and $f$, \eqref{classical-neumann-problem} has at least one radially increasing solution. A similar result has been obtained in \cite{barutello2008note} by applying a shooting method when $a(|x|)f(u)$ is replaced by $|x|^{\alpha}|u|^p$ for every $p>1$ and $\alpha>0$. In \cite{bonheure2012increasing} the authors apply topological and variational arguments to prove that problem \eqref{classical-neumann-problem} admits at least one non-decreasing radial solution without any growth assumption on $f$. They also obtained the existence of a non-constant solution in the case that the problem has nontrivial constant solutions. These results were extended to the $p$-Laplace operator in \cite{colasuonno2016p,secchi2012increasing}. Very recently, Moameni and Salimi \cite{moameni2019existence} considered \eqref{classical-neumann-problem} when $a(|x|)f(u)$ is replaced by $a(|x|)|u|^{p-2}u-b(|x|)|u|^{q-2}u$ and they obtained the existence of a radially non-decreasing solution under some assumptions on $a$ and $b$ without imposing any growth assumption on $p$ and $q$. The method they use relies on a new variational principle recently developed in \cite{moameni2017variational,moameni2018variational} that consists of restricting the Euler-Lagrange function to a convex set. Other interesting papers in which the  Neumann problem of type \eqref{classical-neumann-problem} is considered are (but not limited to)  \cite{yadava1993conjecture,yadava1991existence,cowan2016supercritical,bonheure2016multiple,faraci2003multiplicity}.
    ~~\\
    \item Nonlocal case, i.e., $\cP=(-\Delta)^s$, for $s\in (0,1)$: Cinti and Colasuonno studied in \cite{cinti2020nonlocal} the nonlocal Neumann problem
    \begin{equation}\label{nonlocal-neumann-problem}
        (-\Delta)^su+u=f(u),~~~~~u\geq0~~~~\text{in}~~~\Omega,~~~~~\cN_su=0~~~~\text{in}~~~\R^N\setminus\overline{\Omega}
    \end{equation}
    where $s\in(\frac{1}{2},1)$, $\Omega\subset\R^N$ is either a ball or an annulus and the nonlinearity $f$ is assumed to have supercritical growth. Using a truncation argument, they showed that problem \eqref{nonlocal-neumann-problem} possesses a positive non-decreasing radial solution. In the case when the problem admits a nontrivial constant solution, they also address the existence of a nonconstant solution as in \cite{bonheure2012increasing}.
    ~~~\\
    \item Mixed operator case, i.e., $\cP=\cL=-\Delta+(-\Delta)^s$: As stated above, very few references investigating the Neumann problem governed by mixed local and nonlocal operators $\cL$ can be found in the literature apart from \cite{dipierro2022linear,dipierro2022non,mugnai2022mixed}. As we will see in the rest of the paper, an adaptation can be made to study the Neumann problem of type \eqref{sample} with $\cP=\cL$ and $\cF$ being a supercritical nonlinearity. We will focus on the model case $\cF(u)=a(|x|)|u|^{p-2}u-b(|x|)|u|^{q-2}u$ as in \cite{moameni2019existence}. The consequent lack of compactness can be circumvented by working in a convex set of non-negative and non-decreasing radial functions. Although our argument is closely related to the one in \cite{moameni2019existence}, it involves some technical issues in many regards due to the presence of the nonlocal term. For instance, to the best of our knowledge, no higher regularity is available for the operator $\cL$ with Neumann conditions. Moreover, the action of $\cL$ on radial functions does not give rise to an ODE in contrast with the Laplace operator $-\Delta$.
\end{itemize}
~~\\
The main motivation for considering problem \eqref{e2} in this paper comes from the above discussion. As far as we are aware, there is no result involving a mixed local and nonlocal operator $\cL$ with a supercritical nonlinearity, and with Neumann conditions. We aim to bridge this gap in the present paper. Notice that the case of Dirichlet condition has been considered very recently in \cite{amundsen2023mixed}.
~~~\\\\
To state our main results, we introduce the following assumptions on $a$ and $b$: 

\begin{itemize}
    \item [(A)] $a\in L^1(0,1)$ is non-decreasing and $a(r)>0$ for a.e. $r\in [0,1]$;
    \item [(B)] $b\in L^1(0,1)$ is non-negative and non-increasing in $[0,1]$ 
\end{itemize}
Our first main result is concerned with the case when $p>q$. It reads as follows.

\begin{thm}\label{first-main-result}
	Assume that $(A)$ and $(B)$ hold, and $p>q$. Then, problem \eqref{e2} admits at least one positive non-decreasing radial solution.
\end{thm}

Notice that assumptions $(A)$ and $(B)$ are satisfied by $a=|x|^{\alpha}$ and $b=\frac{\mu}{|x|^{\beta}}$ where $\alpha, \mu\geq0$ and $0\leq\beta<N$. Therefore, Theorem \ref{first-main-result} applies to the following problem

\begin{equation}\label{e2-2}
     \left\{\begin{aligned}
     \cL u+u&=|x|^{\alpha}|u|^{p-2}u-\frac{\mu}{|x|^{\beta}}|u|^{q-2}u\quad\quad~\text{in}~~B_1 \\
     \cN_su&=0 \quad\quad\quad\quad\quad\quad\quad\quad\quad\quad\quad\quad\quad\text{in}~~\R^N\setminus\overline{B}_1\\
     \frac{\partial u}{\partial\nu}&=0 \quad\quad\quad\quad\quad\quad\quad\quad\quad\quad\quad\quad\quad\text{on}~~\partial B_1.
     \end{aligned}
     \right.
	\end{equation}
In particular, when $\mu\equiv0$ and $\alpha>0$, \eqref{e2-2} is the so-called Neumann mixed local and nonlocal H\'{e}non problem. 

Note that when the functions $a$ and $b$ are constants, problem \eqref{e2} always admits a constant solution. In our next result, we show that the solution obtained in Theorem \ref{first-main-result} is non-constant.
\begin{thm}\label{non-consistancy-theorem}
    Suppose that $2<q<p$. Assume also that $a\equiv 1,~ b>0$ and 
    \begin{equation}\label{key-condition-for-non-constancy-of-solutions}
        b(q-2)<p-2.
    \end{equation}
    Then, problem \eqref{e2} admits at least one non-constant non-decreasing radial solution. 
\end{thm}
In Theorem \ref{second-main-result} we address another key result for the case   $p<q$.  In this case, under the integrability assumption  \[\Big(\frac{a^q}{b^p}\Big)^{\frac{1}{q-p}}\in L^1(0,1),\] we prove the existence of one positive non-decreasing radial solution. \\

The paper is organized as follows.  Section \ref{section:proof-of-main-result} is devoted to the proof of existence results for both cases $p<q$ and $p>q.$  In  Section \ref{section:non-constency} we study the non-constancy of solutions.\\

\section{Existence of solutions via a variational method}\label{section:proof-of-main-result}

We begin this section by introducing the functional space in which we work. For all $s\in(0,1)$, we first recall the fractional Sobolev space $H^s_{B_1}$ introduced in \cite{dipierro2017nonlocal} and defined as
\begin{equation}
    H^s_{B_1}:=\Big\{u: \R^N\to\R: u|_{B_1}\in L^2(B_1)~~\text{and}~~\iint_{\cQ}\frac{|u(x)-u(y)|^2}{|x-y|^{N+2s}}\ dxdy<+\infty\Big\},
\end{equation}
where $\cQ=\R^{2N}\setminus(B_1^c)^2$. Now, we define 
\begin{equation}
    \mathbb{X}=\mathbb{X}(B_1)=H^1(B_1)\cap H^s_{B_1}.
\end{equation}
Then $\mathbb{X}$ is a Hilbert space with scalar product
\begin{equation*}
    \langle u, v\rangle_{\mathbb{X}}:=\int_{B_1}uv\ dx+\int_{B_1}\nabla u\cdot \nabla v\ dx+\iint_{\cQ}\frac{(u(x)-u(y))(v(x)-v(y))}{|x-y|^{N+2s}}\ dxdy\quad\text{for all}~~u, v\in \mathbb{X}.
\end{equation*}
Moreover, we also introduce the seminorm $[\cdot]_{\mathbb{X}}$ defined for every $u\in \mathbb{X}$ by
\begin{equation*}
    [u]^2_{\mathbb{X}}=\int_{B_1}|\nabla u|^2\ dx+\iint_{\cQ}\frac{|u(x)-u(y)|^2}{|x-y|^{N+2s}}\ dxdy.
\end{equation*}
Then $\mathbb{X}$ is endowed with the norm
\begin{equation*}
    \|u\|^2_{\mathbb{X}}=\|u\|^2_{L^2(B_1)}+[u]^2_{\mathbb{X}}.
\end{equation*}
Let us also recall the classical fractional Sobolev space $H^s(B_1)$ defined by
\begin{equation*}
    H^s(B_1):=\{u\in L^2(B_1):|u|_{H^s(B_1)}<\infty\}
\end{equation*}
equipped with the norm
\begin{equation*}
    \|u\|_{H^s(B_1)}=\|u\|_{L^2(B_1)}+|u|_{H^s(B_1)},
\end{equation*}
where
\begin{equation*}
    |u|^2_{H^2(B_1)}=\int_{B_1}\int_{B_1}\frac{(u(x)-u(y))^2}{|x-y|^{N+2s}}\ dxdy.
\end{equation*}
Let $\mathbb{X}_{rad}$ be the space of radial functions in $\mathbb{X}$. Denote by $L^p_a(B_1)$ (resp. $L^q_b(B_1)$) the weighted $L^p$ (resp. $L^q$) space defined by
\begin{equation*}
    L^p_a(B_1):=\Big\{u: \int_{B_1}a(|x|)|u|^p\ dx<\infty\Big\}\quad\quad \text{resp.}\quad\quad L^q_b(B_1):=\Big\{u: \int_{B_1}b(|x|)|u|^q\ dx<\infty\Big\}.
\end{equation*}
Now, we consider the Banach space $V$ defined by
\begin{equation*}
    V=\mathbb{X}_{rad}\cap L^p_a(B_1)\cap L^q_b(B_1),
\end{equation*}
equipped with the norm $\|\cdot\|_{V}$ defined by
\begin{equation*}
    \|u\|_{V}:=\|u\|_{\mathbb{X}}+\|u\|_{L^p_a(B_1)}+\|u\|_{L^q_b(B_1)},
\end{equation*}
where 
\begin{equation*}
    \|u\|_{L^p_a(B_1)}:=\Big(\int_{B_1}a(|x|)|u|^p\ dx\Big)^{\frac{1}{p}}\quad\quad\text{and}\quad\quad \|u\|_{L^q_b(B_1)}:=\Big(\int_{B_1}b(|x|)|u|^q\ dx\Big)^{\frac{1}{q}}.
\end{equation*}
Notice that throughout this section, the functions $a$ and $b$ are assumed to satisfy assumptions $(A)$  and $(B)$.

Let $I$ be the Euler-Lagrange functional corresponding to \eqref{e2}. Then, 
\begin{equation*}
    I(u)=\frac{1}{2}\int_{B_1}(|\nabla u|^2+|u|^2)\ dx+\frac{1}{2}\iint_{\cQ}\frac{(u(x)-u(y))^2}{|x-y|^{N+2s}}\ dxdy+\frac{1}{q}\int_{B_1}b(|x|)|u|^q\ dx-\frac{1}{p}\int_{B_1}a(|x|)|u|^p\ dx.
\end{equation*}
We now define $\Psi: V\to\R$ and $\Phi: V\to\R$ by
\begin{equation*}
    \Psi(u):=\frac{1}{2}\int_{B_1}(|\nabla u|^2+|u|^2)\ dx+\frac{1}{2}\iint_{\cQ}\frac{(u(x)-u(y))^2}{|x-y|^{N+2s}}\ dxdy+\frac{1}{q}\int_{B_1}b(|x|)|u|^q\ dx
\end{equation*}
and
\begin{equation*}
    \Phi(u):=\frac{1}{p}\int_{B_1}a(|x|)|u|^p\ dx
\end{equation*}
so that
\begin{equation}
    I:=\Psi-\Phi.
\end{equation}
Clearly, $\Phi\in C^1(V,\R)$ and $\Psi$ is a proper, convex, and lower semi-continuous function. Moreover, $\Psi$ is G$\hat{\text{a}}$teaux differentiable with
\begin{align*}
D\Psi(u)(v)&=\int_{\Omega}(\nabla u\cdot\nabla v+uv)\ dx+\iint_{\cQ}\frac{(u(x)-u(y))(v(x)-v(y))}{|x-y|^{N+2s}}\ dxdy+\int_{B_1}b(|x|)|u|^{q-2}uv\ dx\\
&=\langle u, v\rangle_{\mathbb{X}}+\int_{B_1}b(|x|)|u|^{q-2}uv\ dx.
\end{align*}
Next, we consider the convex set consisting of non-negative, radial, non-decreasing functions
\begin{equation}\label{convex-set}
    K:=\Bigg\{u\in V: \begin{aligned}
        & u~\text{is radial and}~ u\geq0~\text{in}~\R^N,\\
        & u(r_1)\leq u(r_2)~~\text{for all}~~0<r_1\leq r_2<1
    \end{aligned}\Bigg\}.
\end{equation}
We then denote by 
\begin{equation}\label{restricted-euler-lagrange-function}
    I_K=\Psi_K-\Phi
\end{equation}
the restriction of $I$ to $K$, where
\begin{equation}\label{psi-k}
    \Psi_K(u)=\left\{\begin{aligned}
        &\Psi(u)~~~\text{if}~~u\in K,\\
        &+\infty~~~\text{otherwise}.
    \end{aligned}  
    \right.
\end{equation}
We now recall the following definition of a critical point for lower semi-continuous functions due to Szulkin, see \cite{szulkin1986minimax}.

\begin{defi}\label{def3}
	Let $V$ be a real Banach space, $\Phi\in C^1(V,\R)$ and $\psi: V\to (-\infty,+\infty]$ be proper, convex and lower semi-continuous. A point $u\in V$ is a critical point of 
    \begin{equation}\label{energy-functional}
        F=\psi-\Phi
    \end{equation}
    if $u\in Dom(\psi)$ and it satisfies the inequality
	\begin{equation}\label{critical-point}
	\langle D\Phi(u), u-v\rangle+\psi(v)-\psi(u)\geq0,~~\text{for all}~~v\in V.
	\end{equation}
\end{defi}
We also have the following

\begin{defi}
	Let $F=\psi-\Phi$ and $c\in\R$. We say that $F$ satisfies the  Palais-Smale compactness condition \textbf{(PS)} if every sequence $\{u_k\}$ such that $F(u_k)\to c$ and 
	\begin{equation}
	\langle D\Phi(u_k), u_k-v\rangle+\psi(v)-\psi(u_k)\geq-\varepsilon_k\|u_k-v\|_{V},~~~\text{for all}~~v\in V,
	\end{equation}
	where $\varepsilon_k\to0$, possesses a convergent subsequence.
\end{defi}
The next two results are due to Szulkin \cite{szulkin1986minimax}.

\begin{thm}(Mountain pass theorem)\label{mountain-pass-theorem}
	Suppose that $F:V\to (-\infty,+\infty]$ is as in \eqref{energy-functional} and satisfies \textbf{(PS)} and suppose moreover that $F$ satisfies mountain pass geometry \textbf{(MPG)}
   \begin{itemize}
       \item [$(i)$] $F(0)=0$
       \item[$(ii)$] there exists $e\in V$ such that $F(e)\leq 0$
       \item[$(iii)$] there exists some $\eta$ such that $0<\eta<\|e\|$ and for every $u\in V$ with $\|u\|=\eta$ one has $F(u)>0$.
   \end{itemize}
   Then $F$ has a critical value $c$ which is defined by
   \begin{equation}\label{critical-value}
       c=\inf_{\gamma\in\Gamma}\sup_{t\in [0,1]}F(\gamma(t)),
   \end{equation}
   where $\Gamma=\{\gamma\in C([0,1],V): \gamma(0)=0,~ \gamma(1)=e\}$.
\end{thm}

\begin{thm}\label{y1}
    Suppose that $F: V\to (-\infty,+\infty]$ is as in \eqref{energy-functional} and satisfies the Palais-Smale condition. If $F$ is bounded from below, then $c=\inf_{u\in V}F(u)$ is a critical value.
\end{thm}
Inspired by a variational principle in \cite{moameni2018variational}, we have the following result.
\begin{thm}\label{t1}
	Let $V=\mathbb{X}_{rad}\cap L^p_a(B_1)\cap L^q_b(B_1)$, and let $K$ be the convex and weakly closed subset of $V$ defined in \eqref{convex-set}. Suppose the following two assertions hold:
	\begin{itemize}
		\item [$(i)$] The functional $I_K$ has a critical point $\overline{u}\in V$ in the sense of Definition \ref{def3}, and;
		\item [$(ii)$] there exist $\overline{v}\in K$ with $\cN_s\overline{v}=0$ in $\R^N\setminus\overline{B}_1$ and $\frac{\partial\overline{v}}{\partial\nu}=0$ on $\partial B_1$ such that
		\begin{equation}\label{v-tilde-equation}
		\cL\overline{v}+\overline{v}+b(|x|)|\overline{v}|^{q-2}\overline{v}=D\Phi(\overline{u})=a(|x|)|\overline{u}|^{p-2}\overline{u},
		\end{equation}
		in the weak sense namely,
		\begin{align}\label{v-tilde-weak-formulation}
		\nonumber\int_{B_1}\nabla\overline{v}\cdot\nabla\phi\ dx+\int_{B_1}\overline{v}\phi\ dx+\iint_{\cQ}\frac{(\overline{v}(x)-\overline{v}(y))(\phi(x)-\phi(y))}{|x-y|^{N+2s}}\ &dxdy+\int_{B_1}b(|x|)|\overline{v}|^{q-2}\overline{v}\phi\ dx\\
        &=\int_{B_1}D\Phi(\overline{u})\phi\ dx, \quad\forall\phi\in V.
		\end{align}
	\end{itemize}
Then $\overline{u}\in K$ is a weak solution of the equation
\begin{equation}\label{u1}
\cL u+u=a(|x|)|u|^{p-2}u-b(|x|)|u|^{q-2}u
\end{equation}
with Neumann conditions. 
\end{thm}

\begin{proof}
	Since by assumption $(i)$ $\overline{u}$ is a critical point of $I_K$, then by Definition \ref{def3}, we have
	\begin{align*}
	&\frac{1}{2}\int_{B_1}(|\nabla\phi|^2+|\phi|^2)\ dx+\frac{1}{2}\iint_{\cQ}\frac{(\phi(x)-\phi(y))^2}{|x-y|^{N+2s}}\ dxdy+\frac{1}{q}\int_{B_1}b(|x|)|\phi|^q\ dx\\
    &-\frac{1}{2}\int_{B_1}(|\nabla\overline{u}|^2+|\overline{u}|^2)\ dx-\frac{1}{2}\iint_{\cQ}\frac{(\overline{u}(x)-\overline{u}(y))^2}{|x-y|^{N+2s}}\ dxdy-\frac{1}{q}\int_{B_1}b(|x|)|\overline{u}|^q\ dx\\
	&\geq \langle D\Phi(\overline{u}),\phi-\overline{u}\rangle=:\int_{B_1}D\Phi(\overline{u})(\phi-\overline{u})\ dx, \quad\quad\forall\phi\in K.
	\end{align*}
	By taking in particular $\phi=\overline{v}$, the above inequality becomes
    \begin{align}\label{l1}
	\nonumber&\frac{1}{2}\int_{B_1}(|\nabla\overline{v}|^2+|\overline{v}|^2)\ dx+\frac{1}{2}\iint_{\cQ}\frac{(\overline{v}(x)-\overline{v}(y))^2}{|x-y|^{N+2s}}\ dxdy+\frac{1}{q}\int_{B_1}b(|x|)|\overline{v}|^q\ dx\\
    \nonumber&-\frac{1}{2}\int_{B_1}(|\nabla\overline{u}|^2+|\overline{u}|^2)\ dx-\frac{1}{2}\iint_{\cQ}\frac{(\overline{u}(x)-\overline{u}(y))^2}{|x-y|^{N+2s}}\ dxdy-\frac{1}{q}\int_{B_1}b(|x|)|\overline{u}|^q\ dx\\
	&\geq\int_{B_1}D\Phi(\overline{u})(\overline{v}-\overline{u})\ dx.
	\end{align}
	We now use $\phi=\overline{v}-\overline{u}$ as a test function in \eqref{v-tilde-weak-formulation} to get
	\begin{align}\label{l2}
	\nonumber\int_{\Omega}\nabla\overline{v}\cdot\nabla(\overline{v}-\overline{u})\ dx+\int_{B_1}\overline{v}(\overline{v}-\overline{u})&+\iint_{\cQ}\frac{(\overline{v}(x)-\overline{v}(y))((\overline{v}-\overline{u})(x)-(\overline{v}-\overline{u})(y))}{|x-y|^{N+2s}}\ dxdy\\
    &+\int_{B_1}b(|x|)|\overline{v}|^{q-2}\overline{v}(\overline{v}-\overline{u})\ dx=\int_{\Omega}D\Phi(\overline{u})(\overline{v}-\overline{u})\ dx.
	\end{align}
	Plugging \eqref{l2} into \eqref{l1}, we get
	\begin{align}\label{l3}
	\nonumber&\frac{1}{2}\int_{B_1}(|\nabla\overline{v}|^2+|\overline{v}|^2)\ dx+\frac{1}{2}\iint_{\cQ}\frac{(\overline{v}(x)-\overline{v}(y))^2}{|x-y|^{N+2s}}\ dxdy+\frac{1}{q}\int_{B_1}b(|x|)|\overline{v}|^q\ dx\\
    \nonumber&-\frac{1}{2}\int_{B_1}(|\nabla\overline{u}|^2+|\overline{u}|^2)\ dx-\frac{1}{2}\iint_{\cQ}\frac{(\overline{u}(x)-\overline{u}(y))^2}{|x-y|^{N+2s}}\ dxdy-\frac{1}{q}\int_{B_1}b(|x|)|\overline{u}|^q\ dx\\
    \nonumber&\geq\int_{\Omega}\nabla\overline{v}\cdot\nabla(\overline{v}-\overline{u})\ dx+\int_{B_1}\overline{v}(\overline{v}-\overline{u})+\iint_{\cQ}\frac{(\overline{v}(x)-\overline{v}(y))((\overline{v}-\overline{u})(x)-(\overline{v}-\overline{u})(y))}{|x-y|^{N+2s}}\ dxdy\\
    &+\int_{B_1}b(|x|)|\overline{v}|^{q-2}\overline{v}(\overline{v}-\overline{u})\ dx.
	\end{align}
    Note that  $t\mapsto f(t)=\frac{1}{q}|t|^q$ is convex. Therefore, for all $t_1, t_2\in\R$,
    \begin{equation}\label{property-of-covex-functions}
        f(t_2)-f(t_1)\geq f'(t_1)(t_2-t_1)=|t_1|^{q-2}t_1(t_2-t_1).
    \end{equation}
    By substituting $t_1=\overline{v}$ and $t_2=\overline{u}$ in \eqref{property-of-covex-functions}, we have
    \begin{equation*}
        \frac{1}{q}|\overline{u}|^q-\frac{1}{q}|\overline{v}|^q\geq |\overline{v}|^{q-2}\overline{v}(\overline{u}-\overline{v}).
    \end{equation*}
    Then, multiplying the above inequality by $b(|x|)$ and integrating over $B_1$, one gets
    \begin{equation*}
        \frac{1}{q}\int_{B_1}b(|x|)|\overline{u}|^q\ dx-\frac{1}{q}\int_{B_1}b(|x|)|\overline{v}|^q\ dx\geq \int_{B_1}b(|x|)|\overline{v}|^{q-2}\overline{v}(\overline{u}-\overline{v})\ dx.
    \end{equation*}
    Taking this into account, it follows from \eqref{l3} that
    \begin{align}\label{l3'}
	\nonumber&\frac{1}{2}\int_{B_1}|(\nabla\overline{v}|^2+|\overline{v}|^2)\ dx+\frac{1}{2}\iint_{\cQ}\frac{(\overline{v}(x)-\overline{v}(y))^2}{|x-y|^{N+2s}}\ dxdy\\
    \nonumber&-\frac{1}{2}\int_{B_1}(|\nabla\overline{u}|^2+|\overline{u}|^2)\ dx-\frac{1}{2}\iint_{\cQ}\frac{(\overline{u}(x)-\overline{u}(y))^2}{|x-y|^{N+2s}}\ dxdy\\
    &\geq\int_{B_1}\nabla\overline{v}\cdot\nabla(\overline{v}-\overline{u})\ dx+\int_{B_1}\overline{v}(\overline{v}-\overline{u})+\iint_{\cQ}\frac{(\overline{v}(x)-\overline{v}(y))((\overline{v}-\overline{u})(x)-(\overline{v}-\overline{u})(y))}{|x-y|^{N+2s}}\ dxdy.
	\end{align}
    Using the elementary identity
    \begin{align*}
        ((\overline{v}-\overline{u})(x)-(\overline{v}-\overline{u})(y))^2&=-(\overline{v}(x)-\overline{v}(y))^2+(\overline{u}(x)-\overline{u}(y))^2\\
        &+2(\overline{v}(x)-\overline{v}(y))((\overline{v}-\overline{u})(x)-(\overline{v}-\overline{u})(y)),
    \end{align*}
    it follows from \eqref{l3'} that
	\begin{equation*}
	\frac{1}{2}\int_{B_1}|\nabla(\overline{v}-\overline{u})|^2\ dx+\frac{1}{2}\int_{B_1}|\overline{v}-\overline{u}|^2\ dx+\frac{1}{2}\iint_{\cQ}\frac{((\overline{v}-\overline{u})(x)-(\overline{v}-\overline{u})(y))^2}{|x-y|^{N+2s}}\ dxdy\leq0
	\end{equation*}
	and therefore, $\overline{v}-\overline{u}=0$ i.e., $\overline{v}=\overline{u}$ a.e., in $\R^N$. We then deduce from \eqref{v-tilde-equation} that $\overline{u}$ satisfies in the weak sense equation \eqref{u1}. Moreover, since $\overline{v}$ satisfies Neumann conditions, so does $\overline{u}$. The proof is therefore finished. 
\end{proof}
We now collect a series of lemmas that will play a key role in the proof of our main results. Let us start with the following.
\begin{lemma}\label{llm}
The following assertions hold:
    \begin{itemize}
        \item [$(i)$] There exists a constant $C_*>0$ such that
        \begin{equation*}
            \|u\|_{L^{\infty}(B_1)}\leq C_{*}\|u\|_{\mathbb{X}},\quad\quad\text{for all}~~u\in K;
        \end{equation*}
        \item[$(ii)$] There exists a constant $C>0$ such that
        \begin{equation}
            \|u\|_{\mathbb{X}}\leq\|u\|_{V}\leq C\|u\|_{\mathbb{X}},\quad\quad\text{for all}~~ u\in K.
        \end{equation}
    \end{itemize}
\end{lemma}
\begin{proof}
    We start by proving assertion $(i)$. We follow the lines of that of \cite[Lemma 4.3]{brasco2016global}. Let $r_0<\rho<1$. Utilising the fact that $u$ is radial, $s>\frac{1}{2}$, and the trace inequality for $H^s(B_{\rho}\setminus  B_{r_0})$ (see e.g. \cite[Section 3.3.3]{serra2011monotonicity}), it follows that for every $x\in\partial B_{\rho}$,
    \begin{align*}
    |u(x)|^2&=\frac{\rho^{1-N}}{N\omega_N}\int_{\partial B_{\rho}}|u|^2\ d\cH^{N-1}\\
    &\leq c\frac{\rho^{1-N}}{N\omega_N}\rho^{2s-1}\Big(|u|^2_{H^s(B_{\rho}\setminus B_{r_0})}+\frac{1}{\rho^{2s}}\|u\|^2_{L^2(B_{\rho}\setminus B_{r_0})}\Big),
    \end{align*}
    where $\omega_N$ is the volume of the $N$-dimensional unit sphere in $\R^N$ and $d\cH^{N-1}$ denotes the $(N-1)$-dimensional Hausdorff measure. The above inequality reduces to
    \begin{align*}
        |u(x)|&\leq c\rho^{-s}|x|^{\frac{2s-N}{2}}\|u\|_{H^s(B_{\rho}\setminus B_{r_0})}\quad\quad\text{if}~~\rho=|x|\\
        &\leq c(1+r_0^{-s})r_0^{\frac{2s-N}{2}}\|u\|_{H^s(B_1\setminus B_{r_0})}
    \end{align*}
    Hence, 
    \begin{equation}\label{u2}
        \|u\|_{L^{\infty}(B_1\setminus B_{r_0})}\leq c_{r_0}\|u\|_{H^s(B_1\setminus B_{r_0})}\leq c_{r_0}\|u\|_{H^s(B_1)}\leq c_{r_0}\|u\|_{\mathbb{X}}
    \end{equation}
    where $c_{r_0}:=c(1+r_0^{-s})r_0^{\frac{2s-N}{2}}$. Now, since $u$ is radially non-decreasing, then $\|u\|_{L^{\infty}(B_1)}=\|u\|_{L^{\infty}(B_1\setminus B_{\frac{1}{2}})}$. Taking this into account, we substitute $r_0=\frac{1}{2}$ in \eqref{u2} to get that
    \begin{equation*}
        \|u\|_{L^{\infty}(B_1)}\leq C_*\|u\|_{\mathbb{X}},
    \end{equation*}
    where $C_*=c_{\frac{1}{2}}$. This completes the proof of assertion $(i)$.

    Regarding the proof of $(ii)$, the inequality $\|u\|_{\mathbb{X}}\leq \|u\|_{V}$ is trivial by definition of the norm $\|\cdot\|_{V}$. Now, to obtain the second inequality, we use $(i)$ to see that
    \begin{align*}
        \|u\|_{V}&=\|u\|_{\mathbb{X}}+\|u\|_{L^p_a(B_1)}+\|u\|_{L^q_b(B_1)}\\
        &\leq\|u\|_{\mathbb{X}}+\|u\|_{L^{\infty}(B_1)}\Big(\|a\|^{\frac{1}{p}}_{L^1(B_1)}+\|b\|^{\frac{1}{q}}_{L^1(B_1)}\Big)\\
        &\leq C\|u\|_{\mathbb{X}}
    \end{align*}
    where $C=1+C_{*}\Big(\|a\|^{\frac{1}{p}}_{L^1(B_1)}+\|b\|^{\frac{1}{q}}_{L^1(B_1)}\Big)$, as desired.
\end{proof}

\begin{lemma}\label{mpg}
    The functional $I_K$ defined in \eqref{restricted-euler-lagrange-function} satisfies the \textbf{(PS)} condition if either of the following assertions hold:
    \begin{itemize}
        \item [$(i)$] $p>q$;
        \item[$(ii)$] If $p<q$ then $\Big(\frac{a^q}{b^p}\Big)^{\frac{1}{q-p}}\in L^1(0,1)$.
    \end{itemize}
\end{lemma}
\begin{proof}
    We follow the lines of that of \cite[Lemma 3.4]{moameni2019existence}. Let $\{u_k\}\subset K$ be a sequence such that $I_K(u_k)\to c\in \R$ and
    \begin{equation}\label{u7}
        \langle D\Phi(u_k), u_k-v\rangle+\Psi_K(v)-\Psi_K(u_k)\geq-\varepsilon_k\|u_k-v\|_{V},~~~\forall v\in V.
    \end{equation}
    We shall show that $\{u_k\}$ possesses a convergent subsequence in $V$. For that, we first show that $u_k$ is bounded in $V$. We distinguish the cases $p>q$ and $p<q$.

    \textbf{Case 1.} $p>q$. Since $I_K(u_k)\to c$, then for $k$ sufficiently large, one has
    \begin{equation}\label{u8}
        \frac{1}{2}\|u_k\|^2_{\mathbb{X}}+\frac{1}{q}\|u_k\|^q_{L^q_b(B_1)}-\frac{1}{p}\int_{B_1}a(|x|)|u_k|^p\ dx\leq c+1.
    \end{equation}
    We now introduce the function $g(t)=t^q-p(t-1)-1$ for $t\in (1,+\infty)$. Then, there exists $\ov t=(\frac{p}{q})^{\frac{q}{q-1}}$ such that for every $t\in(1,\ov t)$, $g(t)<0$. Henceforth, by choosing such a number, we have $t>1$ and $t^{q}-1<p(t-1)$.

    Now, substituting $v=tu_k$ in \eqref{u7} and recalling that $\langle D\Phi(u_k), u_k\rangle=\int_{B_1}a(|x|)u_k(x)^p\ dx$, we have that
    \begin{equation}\label{u9-1}
        \frac{(1-t^2)}{2}\|u\|^2_{\mathbb{X}}+\frac{(1-t^q)}{q}\|u_k\|^q_{L^q_b(B_1)}+(t-1)\int_{B_1}a(|x|)|u_k|^p\ dx\leq \varepsilon_k (t-1)\|u_k\|_{V}\leq C\|u_k\|_{V}.
    \end{equation}
    Recalling that $t^{q}-1<p(t-1)$, then we let $\beta>0$ so that
    \begin{equation*}
        \frac{1}{p(t-1)}<\beta<\frac{1}{t^q-1}.
    \end{equation*}
    Multiplying \eqref{u9-1} and summing it up with \eqref{u8}, we have
    \begin{equation*}
        \frac{1+\beta(1-t^2)}{2}\|u_k\|^2_{\mathbb{X}}+\frac{1+\beta(1-t^q)}{q}\|u_k\|^q_{L^q_b(B_1)}+\Big(\beta(t-1)-\frac{1}{p}\Big)\int_{B_1}a(|x|)|u_k|^p\ dx\leq c+1+\beta c\|u_k\|_{V}.
    \end{equation*}
    By the choice of $\beta$, the fact that $t>1$ and $q\geq2$ then 
    the constants on the left-hand side of the inequality above are positive. Therefore,
    \begin{equation*}
        \|u_k\|^2_{\mathbb{X}}+C_1\|u_k\|^q_{L^q_b(B_1)}+C_2\int_{B_1}a(|x|)|u_k|^p\ dx\leq C_3+C_4\|u_k\|_{V}.
    \end{equation*}
     for suitable constants $C_i>0,~i\in\{1,2,3,4\}$. From assertion $(ii)$ of Lemma \ref{llm}, we have
     \begin{equation*}
         \|u_k\|^2_{\mathbb{X}}\leq C_3+C_4\|u_k\|_{V}\leq C_3+cC_4\|u_k\|_{\mathbb{X}}
     \end{equation*}
     and thus, $\{u_k\}$ is bounded in $\mathbb{X}$. We then conclude that $\{u_k\}$ is also bounded in $V$.

     \textbf{Case 2.} $p<q$. Using the H\"{o}lder inequality with exponents $\frac{q}{q-p}$ and $\frac{q}{p}$, we have
     \begin{align*}
         \int_{B_1}a(|x|)|u_k|^p\ dx&=\int_{B_1}a(|x|)b(|x|)^{-\frac{p}{q}} b(|x|)^{\frac{p}{q}}|u_k|^p\ dx\\
         &\leq \|a\cdot b^{-\frac{p}{q}}\|_{L^{\frac{q}{q-p}}(B_1)}\Big(\int_{B_1}b(|x|)|u_k|^q\Big)^{\frac{p}{q}}.
     \end{align*}
     Hence,
     \begin{align}\label{k1}
      \nonumber I_K(u_k)&\geq \frac{1}{2}\int_{B_1}(|\nabla u_k|^2+|u|^2)\ dx+\frac{1}{2}\iint_{\cQ}\frac{(u_k(x)-u_k(y))^2}{|x-y|^{N+2s}}\ dxdy+\frac{1}{q}\int_{B_1}b(|x|)|u_k|^q\ dx\\
    \nonumber&~~~~~~~~~~~~~~~~~~~~~~~~~~-\frac{1}{p}\|a\cdot b^{-\frac{p}{q}}\|_{L^{\frac{q}{q-p}}(B_1)}\Big(\int_{B_1}b(|x|)|u_k|^q\Big)^{\frac{p}{q}}\\
         &=\frac{1}{2}\|u_k\|^2_{\mathbb{X}}+\frac{1}{q}\|u_k\|^q_{L^{q}_b(B_1)}-\frac{1}{p}\|a\cdot b^{-\frac{p}{q}}\|_{L^{\frac{q}{q-p}}(B_1)}\|u_k\|^p_{L^q_b(B_1)}.
     \end{align}
     Since by assumption $p<q$ and $\|a\cdot b^{-\frac{p}{q}}\|_{L^{\frac{q}{q-p}}(B_1)}<\infty$, then we deduce from the inequality above that $I_K$ is bounded from below and coercive in $V$ thanks to Lemma \ref{llm} $(ii)$. Therefore, we deduce that $\{u_k\}$ is bounded in $V$, as desired.

     So, the sequence $\{u_k\}$ is bounded in $V$ for both $p>q$ and $p<q$. Hence, up to a subsequence, there exists $\overline{u}\in V$ such that $u_k\rightharpoonup \overline{u}$ weakly in $\mathbb{X}$ and $u_k\to \overline{u}$ strongly in $L^2(B_1)$ thanks to the compact embedding $\mathbb{X}\hookrightarrow L^2(B_1)$. In particular, $u_k\to \overline{u}$ a.e. in $B_1$. Moreover, we also have $u_k\rightharpoonup \overline{u}$ weakly in $L^p_a(B_1)$ resp. $L^q_b(B_1)$. Notice that from Lemma \ref{llm} $(i)$, the sequence $\{u_k\}$ is also bounded in $L^{\infty}(B_1)$. Moreover, since each $u_k$ is radial, then $\overline{u}$ is radial as well. Thus, $\overline{u}\in K$.

     We now wish to prove that
     \begin{equation}\label{u9}
         u_k\to \overline{u}\quad\text{strongly in}~~V.
     \end{equation}
     For that, we mention first that by the lower semi-continuity property, one has
     \begin{align}\label{u10}
         \|\overline{u}\|_{\mathbb{X}}\leq \liminf_{k\to\infty}\|u_k\|_{\mathbb{X}}\quad\quad\text{and}\quad\quad \|\overline{u}\|_{L^q_b(B_1)}\leq \liminf_{k\to\infty}\|u_k\|_{L^q_b(B_1)}.
     \end{align}
     Now, taking $v=\overline{u}$ in \eqref{u7} gives
     \begin{align}\label{u11}
        \frac{1}{2}(\|\overline{u}\|^2_{\mathbb{X}}-\|u_k\|^2_{\mathbb{X}})+\frac{1}{q}(\|\overline{u}\|^q_{L^q_b(B_1)}-\|u_k\|^q_{L^q_b(B_1)})+\int_{B_1}a(|x|)|u_k|^{p-2}u_k(u_k-\overline{u})\ dx\geq-\varepsilon_k\|u_k-\overline{u}\|_{V}. 
     \end{align}
     Since $\|u_k\|_{L^{\infty}(B_1)}$ is bounded, then
     \begin{equation*}
         \left\{\begin{aligned}
             a(|x|)|u_k|^{p-1}|u_k-\overline{u}|&\leq a(|x|)\|u_k\|^{p-1}_{L^{\infty}(B_1)}(\|u_k\|_{L^{\infty}(B_1)}+\|\overline{u}\|_{L^{\infty}(B_1)})\leq C a(|x|),~~\text{and}\\
             a(|x|)|u_k|^p&\leq a(|x|)\|u_k\|^p_{L^{\infty}(B_1)}\leq Ca(|x|)
         \end{aligned}
         \right.
     \end{equation*}
     for some $C>0$. Now, using that $a\in L^1$, then the dominated convergence theorem yields
     \begin{align}\label{u12}
        \lim_{k\to\infty} \int_{B_1}a(|x|)|u_k|^{p-2}u_k(u_k-\overline{u})\ dx=0~~~\text{and}~~~\lim_{k\to\infty}\int_{B_1}a(|x|)|u_k|^p\ dx=\int_{B_1}a(|x|)|\overline{u}|^p\ dx.
     \end{align}
     From \eqref{u11} and \eqref{u12}, it follows that
     \begin{equation}\label{u13}
         \frac{1}{2}\Big(\limsup_{k\to\infty}\|u_k\|^2_{\mathbb{X}}-\|\overline{u}\|^2_{\mathbb{X}}\Big)+\frac{1}{q}\Big(\limsup_{k\to\infty}\|u_k\|^q_{L^q_b(B_1)}-\|\overline{u}\|^q_{L^q_b(B_1)}\Big)\leq0.
     \end{equation}
     Combining \eqref{u13} together with \eqref{u10}, we deduce that
     \begin{equation}\label{u14}
         \lim_{k\to\infty}\|u_k\|^2_{\mathbb{X}}=\|\overline{u}\|^2_{\mathbb{X}}\quad\quad\text{and}\quad\quad \lim_{k\to\infty}\|u_k\|^q_{L^q_b(B_1)}=\|\overline{u}\|^q_{L^q_b(B_1)}.
     \end{equation}
     Finally, from \eqref{u14} and \eqref{u12} we obtain \eqref{u9}, as desired.
\end{proof}
The next lemma plays a key role in establishing assertion $(ii)$ in Theorem \ref{t1}.
\begin{lemma}\label{lllm}
    Let $g\in L^1(0,1)$ be a non-negative monotone function. Then there exists a sequence of smooth monotone functions $\{g_m\}$ with the property that $0\leq g_m\leq g$ and $g_m\to g$ strongly in $L^1(0,1)$.
\end{lemma}
\begin{proof}
    The proof of this lemma can be found in \cite{cowan2017existence}. We sketch it here for the sake of completeness. Without loss of generality, we focus on the case when $g$ is monotone non-decreasing. The same argument works if $g$ is monotone non-increasing.
    
    For large $m$, we define $[0,\infty)\ni r\mapsto q_m(r)=\min\{g(r),m\}$. Then, $q_m$ is non-decreasing in $(0,1)$ for large $m$. We now extend $q_m(r)$ to $q_m(1)$ for $r>1$ and $q_m=0$ for $r<0$. Let $\eta\geq0$ be a smooth function with $\eta=0$ on $(-\infty,-1)\cup (0,\infty)$ and $\eta>0$ on $(-1,0)$. Assume also that $\int_{-1}^{0}\eta(t)\ dt=1$.
    For $\varepsilon>0$, we set $\eta_{\varepsilon}(r)=\frac{1}{\varepsilon}\eta(\frac{r}{\varepsilon})$ and define
    \begin{equation*}
        q^{\varepsilon}_m(r):=\int_{-\varepsilon}^{0}\eta_{\varepsilon}(t)q_m(r+t)\ dt. 
    \end{equation*}
    Since $q_m$ is non-decreasing, it follows that for every fixed small $\varepsilon>0$, $q^{\varepsilon}_m$ is non-decreasing in $r$. Moreover, notice that
    \begin{equation*}
        0\leq q^{\varepsilon}_m(r)=\int_{-\varepsilon}^{0}\eta_{\varepsilon}(t)q_m(r+t)\ dt\leq q_m(r)\int_{-\varepsilon}^{0}\eta_{\varepsilon}(t)\ dt=q_m(r)\leq g(r).
    \end{equation*}
    We now let $\varepsilon_m\searrow0$ and we define $g_m(r):=q_m^{\varepsilon_m}(r)$. Thus, from the inequality above, we have $0\leq g_m(r)\leq g(r)$ for all $m$. Moreover, $r\mapsto g_m(r)$ is non-decreasing. Finally, one can now show that $g_m\to g$ in $L^1(0,1)$.
\end{proof}

\begin{lemma}\label{llllm}
    Suppose that $\overline{u}\in K$. Then there exists $v\in K$ solving
    \begin{equation}\label{e4}
     \left\{\begin{aligned}
     \cL v+v+b(|x|)|v|^{q-2}v&=a(|x|)|\overline{u}|^{p-2}\overline{u}\quad\quad~\text{in}~~B_1 \\
     \cN_sv&=0 \quad\quad\quad\quad\quad\quad\quad~\text{in}~~\R^N\setminus\overline{B}_1\\
     \frac{\partial v}{\partial\nu}&=0 \quad\quad\quad\quad\quad\quad\quad~\text{on}~~\partial B_1.
     \end{aligned}
     \right.
	\end{equation}
 in the weak sense.
\end{lemma}
This lemma is the mixed local and nonlocal version of \cite[Lemma 3.6]{moameni2019existence}. As we will show below, to get the monotonicity property of $v$, we use a different strategy to that of \cite[Lemma 3.6]{moameni2019existence}. In fact, in \cite[Lemma 3.6]{moameni2019existence}, the authors take advantage of the fact that the regularity of $v$ can be upgraded to $H^3(B_1)$ so that the equation can be differentiated with respect to the radial variable. Notice also that toward this, they used the fact that the equation can be transformed into an ODE. However, we do not have such properties in the present case.
\begin{proof}
    According to Lemma \ref{lllm}, there exists a sequence $\{a_m\}$ (resp. $\{b_m\}$) of smooth functions such that $0\leq a_m\leq a$ (resp. $0\leq b_m\leq b$) and every $a_m$ is non-decreasing (resp. every $b_m$ is non-increasing) on $(0,1)$ and $a_m\to a$ strongly in $L^1(0,1)$ (resp. $b_m\to b$ strongly in $L^1(0,1)$). We now consider the equation
    \begin{equation}\label{e4-1}
     \left\{\begin{aligned}
     \cL v+v+b_m(|x|)|v|^{q-2}v&=a_m(|x|)|\overline{u}|^{p-2}\overline{u}\quad\quad~\text{in}~~B_1 \\
     \cN_sv&=0 \quad\quad\quad\quad\quad\quad\quad~~~~\text{in}~~\R^N\setminus\overline{B}_1\\
     \frac{\partial v}{\partial\nu}&=0 \quad\quad\quad\quad\quad\quad\quad~~~~\text{on}~~\partial B_1.
     \end{aligned}
     \right.
	\end{equation}
 Let $J: \mathbb{X}_{rad}\to \R$ be the functional defined by
 \begin{align*}
     J(v)&:=\frac{1}{2}\int_{B_1}(|\nabla v|^2+|v|^2)\ dx+\frac{1}{2}\iint_{\cQ}\frac{(v(x)-v(y))^2}{|x-y|^{N+2s}}\ dxdy\\
     &~~~~~~~~~~~+\frac{1}{q}\int_{B_1}b_m(|x|)|v|^q\ dx-\int_{B_1}f_m(\overline{u})v\ dx,
 \end{align*}
 where $f_m(\ov u)=a_m(|x|)|\overline{u}|^{p-2}\overline{u}$. Then, $J$ is convex, lower semi-continuous and 
 \begin{equation*}
     \lim_{\|v\|_{\mathbb{X}}\to \infty}J(v)=\infty.
 \end{equation*}
 Hence, $J$ achieves its minimum at some $v_m\in \mathbb{X}_{rad}$. Moreover, $v_m$ satisfies the Neumann conditions 
 \begin{equation}\label{nc}
     \cN_s v_m=0~~~\text{in}~~\R^N\setminus\overline{B}_1\quad\quad\text{and}\quad\quad \frac{\partial v_m}{\partial\nu}=0~~~\text{on}~~\partial B_1.
 \end{equation}
 Indeed, since $v_m$ is the minimum of $J$ in $\mathbb{X}_{rad}$, it is also a critical point of $J$ in $\mathbb{X}_{rad}$. Therefore, it satisfies
 \begin{align}\label{z1}
   \nonumber \int_{B_1}\nabla v_m\cdot \nabla\phi\ dx&+\int_{B_1}v_m\phi\ dx+\iint_{\cQ}\frac{(v_m(x)-v_m(y))(\phi(x)-\phi(y))}{|x-y|^{N+2s}}\ dxdy\\
     &+\int_{B_1}b_m(|x|)|v_m|^{q-2}v_m\phi\ dx=\int_{B_1}f_m(\overline{u})\phi\ dx,\quad\quad \forall\phi\in \mathbb{X}.
 \end{align}
 We now take $\phi\equiv0$ in $\overline{B}_1$ and from \eqref{z1}, it follows that
 \begin{align*}
     0=\iint_{\cQ}\frac{(v_m(x)-v_m(y))(\phi(x)-\phi(y))}{|x-y|^{N+2s}}\ dxdy&=2\int_{\R^N\setminus \overline{B}_1}\Bigg(\int_{B_1}\frac{v_m(x)-v_m(y)}{|x-y|^{N+2s}}\ dy\Bigg)\phi(x)\ dx\\
     &=2\int_{\R^N\setminus\overline{B}_1}\cN_s v_m(x)\phi(x)\ dx,\quad\quad\forall\phi.
 \end{align*}
 In particular, $\cN_sv_m=0$ a.e. in $\R^N\setminus\overline{B}_1$. On the other hand, from the integration by parts formula
 \begin{equation}\label{integration-by-parts}
     \left\{\begin{aligned}
         \iint_{\cQ}\frac{(v_m(x)-v_m(y))(\phi(x)-\phi(y))}{|x-y|^{N+2s}}\ dxdy&=\int_{B_1}(-\Delta)^sv_m\phi\ dx+\int_{\R^N\setminus B_1}\cN_s v_m\phi\ dx;\\
         \int_{B_1}\nabla v_m\cdot\nabla\phi\ dx&=\int_{B_1}(-\Delta)v_m\phi\ dx+\int_{\partial B_1}\frac{\partial v_m}{\partial\nu}\phi\ d\sigma,\quad\quad\forall\phi,
     \end{aligned}
     \right.
 \end{equation}
 we deduce from \eqref{z1} that
 \begin{equation}\label{d-1}
     \int_{B_1}(\cL v_m+v_m+b_{m}(|x|)|v_m|^{q-2}v_m-f_m(\overline{u}))\phi\ dx+\int_{\R^N\setminus B_1}\cN_sv_m\phi\ dx+\int_{\partial B_1}\frac{\partial v_m}{\partial\nu}\phi\ d\sigma=0~~~\forall\phi. 
 \end{equation}
 In particular, by taking $\phi\equiv0$ in $\R^N\setminus\overline{B}_1$, \eqref{d-1} becomes
 \begin{equation*}
     \int_{B_1}(\cL v_m+v_m+b_{m}(|x|)|v_m|^{q-2}v_m-f_m(\overline{u}))\phi\ dx=0
 \end{equation*}
 and thus
 \begin{equation}\label{d-2}
     \cL v_m+v_m+b_{m}(|x|)|v_m|^{q-2}v_m=f_m(\overline{u})~~\text{a.e. in}~~B_1.
 \end{equation}
 Now, using \eqref{d-2}, the integration by parts formula \eqref{integration-by-parts}, and the fact that $\cN_s v_m=0$ in $\R^N\setminus \overline{B}_1$, we deduce from \eqref{z1} that
 \begin{equation*}
     \int_{\partial B_1}\frac{\partial v_m}{\partial\nu}\phi\ d\sigma=0\quad\quad\forall\phi,
 \end{equation*}
 and therefore, $\frac{\partial v_m}{\partial\nu}=0$ a.e. on $\partial B_1$.

 We now wish to show that $v_m\in K$. Notice first that $v_m\in\mathbb{X}\cap L^q_{b_m}(B_1)$ satisfies
 \begin{align}\label{z2}
   \nonumber &\int_{B_1}\nabla v_m\cdot \nabla\phi\ dx+\iint_{\cQ}\frac{(v_m(x)-v_m(y))(\phi(x)-\phi(y))}{|x-y|^{N+2s}}\ dxdy\\
     &~~~+\int_{B_1}v_m\phi\ dx+\int_{B_1}b_m(|x|)|v_m|^{q-2}v_m\phi\ dx=\int_{B_1}f_m(\overline{u})\phi\ dx,\quad\forall\phi\in \mathbb{X}\cap L^q_{b_m}(B_1).
 \end{align}
 Since $a$ is smooth and $\overline{u}\in K$, then from the first part of Lemma \ref{llm}, $f_m(\overline{u})\in L^{\infty}(B_1)$. Notice also that $f_m(\overline{u})\geq0$. We now claim that $v_m\geq0$. Indeed, to see this, we use $\phi=v_m^-$ as a test function in \eqref{z2} to get 
    \begin{align}\label{u4}
       \nonumber&\int_{B_1}\nabla v_m\cdot\nabla v_m^-\ dx+\iint_{\cQ}\frac{(v_m(x)-v_m(y))(v_m^-(x)-v_m^-(y))}{|x-y|^{N+2s}}\ dxdy\\
        &~~~~~~~~~~~~~~~~~+\int_{B_1}v_mv_m^-\ dx+\int_{B_1}b(x)|v_m|^{q-2}v_mv_m^-\ dx=\int_{B_1}f_m(\overline{u})u_m^-\ dx.
    \end{align}
    Since
    \begin{equation*}
        \left\{\begin{aligned}
            & \nabla v_m\cdot\nabla v_m^-=-|\nabla v_m^-|^2,~~ v_mv_m^-=-|v_m^-|^2~~\text{and}\\
            & (v_m(x)-v_m(y))(v_m^-(x)-v_m^-(y))\leq -(v_m^-(x)-v_m^-(y))^2,
        \end{aligned}
        \right.
    \end{equation*}
    then it follows from \eqref{u4} that
    \begin{align}\label{u5}
       \nonumber&\int_{B_1}|\nabla v_m^-|^2\ dx+\iint_{\cQ}\frac{(v_m^-(x)-v_m^-(y))2}{|x-y|^{N+2s}}\ dxdy\\
        &~~~~~~~~~~~~~~~~~+\int_{B_1}|v_m^-|^2\ dx+\int_{B_1}b(x)|v_m^-|^{q}\ dx\leq-\int_{B_1}f_m(\overline{u})v_m^-\ dx.
    \end{align}
    Since $f_m(\overline{u})\geq0$ and $b$ is non-negative by assumption, we deduce from \eqref{u5} that $v_m^-\equiv0$ a.e. that is $v_m\geq0$ a.e. in $B_1$. 
 
 Since $v_m$ is radial, then to have $v_m\in K$, it suffices to show that $v_m$ is non-decreasing with respect to the radial variable. For that, we follow some ideas in \cite{colasuonno2016p}. Fix $\ov r\in (0,1)$. Then to prove that $v_m$ is non-decreasing, it is enough\footnote{Indeed, if $v_m(t_0)>v_m(r)$ for some $\ov r< t_0<r$, then by continuity of $v_m$, there exists $t\in (t_0,r)$ for which $v_m(t_0)>v_m(t)>v_m(r)$. This violates both $(i)$ and $(ii)$.} to show that for every $r\in(\ov r,1)$, one of the following cases occurs:
 \begin{itemize}
     \item [$(i)$] $v_m(t)\leq v_m(r)$~~~\text{for all}~~~$t\in (\ov r,r)$
     \item[$(ii)$] $v_m(t)\geq v_m(r)$~~~\text{for all}~~~$t\in (r,1)$.
 \end{itemize}
 Now, if $f_m(\overline{u}(r))\leq v_m(r)+b_m(r)|v_m(r)|^{q-2}v_m(r)$, then we use
 \begin{equation*}
     \phi(x)=\left\{\begin{aligned}
         & (v_m(|x|)-v_m(r))^+~~~\quad\text{if}~~\ov r<|x|\leq r\\
         & 0\quad\quad\quad\quad\quad\quad\quad\quad\quad\quad\text{otherwise}
     \end{aligned}
     \right.
 \end{equation*}
 as a test function in \eqref{z2} to see that
\begin{align}\label{z3}
   \nonumber &\int_{B_r\setminus B_{\ov r}}\nabla v_m\cdot \nabla\phi\ dx+\iint_{\R^{2N}\setminus ((B_r\setminus B_{\ov r})^c)^2}\frac{(v_m(x)-v_m(y))(\phi(x)-\phi(y))}{|x-y|^{N+2s}}\ dxdy\\
     &~~~~~~~~~~~~+\int_{B_r\setminus B_{\ov r}}v_m\phi\ dx+\int_{B_r\setminus B_{\ov r}}b_m(|x|)|v_m|^{q-2}v_m\phi\ dx=\int_{B_r\setminus B_{\ov r}}f_m(\overline{u})\phi\ dx.
 \end{align}
 Now using that $\nabla v_m\cdot\nabla\phi=\nabla (v_m-v_m(r))\cdot\nabla\phi=|\nabla\phi|^2$, $v_m\phi=(v_m-v_m(r))\phi+v_m(r)\phi=|\phi|^2+v_m(r)\phi$ and from the elementary inequality
 \begin{equation*}
     (v_m(x)-v_m(y))(\phi(x)-\phi(y))=((v_m-v_m(r))(x)-(v_m-v_m(r))(y))(\phi(x)-\phi(y))\geq (\phi(x)-\phi(y))^2,
 \end{equation*}
there holds
\begin{align}
&\iint_{\R^{2N}\setminus ((B_r\setminus B_{\ov r})^c)^2}\frac{(v_m(x)-v_m(y))(\phi(x)-\phi(y))}{|x-y|^{N+2s}}\ dxdy\geq \iint_{\R^{2N}\setminus ((B_r\setminus B_{\ov r})^c)^2}\frac{(\phi(x)-\phi(y))^2}{|x-y|^{N+2s}}\ dxdy,\label{i1}\\
&\int_{B_r\setminus B_{\ov r}}\nabla v_m\cdot\nabla\phi\ dx=\int_{B_r\setminus B_{\ov r}}|\nabla\phi|^2\ dx,\label{i2}\\
&\int_{B_r\setminus B_{\ov r}}v_m\phi\ dx=\int_{B_r\setminus B_{\ov r}}|\phi|^2\ dx+\int_{B_r\setminus B_{\ov r}}v_m(r)\phi\ dx\label{i3}.
\end{align}
Next, using that $b$ is non-increasing, then
\begin{align}\label{z4}
   \nonumber \int_{B_r\setminus B_{\ov r}}b_m(|x|)|v_m|^{q-2}v_m\phi\ dx&\geq b_m(r)\int_{B_r\setminus B_{\ov r}}|v_m|^{q-2}v_m\phi\ dx\\
    &=b_m(r)\int_{B_r\setminus B_{\ov r}}(|v_m|^{q-2}v_m-|v_m(r)|^{q-2}v_m(r))\phi\ dx\\
 \nonumber &+\int_{B_r\setminus B_{\ov r}}b_m(r)|v_m(r)|^{q-2}v_m(r)\phi\ dx.
\end{align}
Let $f:[0,1]\to\R$ be given by
\begin{equation*}
    f(t)=|tv_m(x)+(1-t)v_m(r)|^{q-2}(tv_m(x)+(1-t)v_m(r)).
\end{equation*}
Then,
\begin{equation}
    |v_m|^{q-2}v_m-|v_m(r)|^{q-2}v_m(r)=f(1)-f(0)=\int_{0}^{1}f'(t)\ dt
\end{equation}
Now, a direct calculation shows that
\begin{equation*}
    f'(t)=(q-1)|tv_m(x)+(1-t)v_m(r)|^{q-2}(v_m(x)-v_m(r)).
\end{equation*}
So,
\begin{align}\label{p}
  \nonumber  &\int_{B_r\setminus B_{\ov r}}(|v_m|^{q-2}v_m-|v_m(r)|^{q-2}v_m(r))\phi\ dx\\
  \nonumber  &=(q-1)\int_{B_r\setminus B_{\ov r}}\Bigg(\int_{0}^{1}|tv_m(x)+(1-t)v_m(r)|^{q-2}\ dt\Bigg)(v_m(x)-v_m(r))\phi\ dx\\
    &=(q-1)\int_{B_r\setminus B_{\ov r}}\Bigg(\int_{0}^{1}|tv_m(x)+(1-t)v_m(r)|^{q-2}\ dt\Bigg)|\phi|^2\ dx.
\end{align}
Plugging \eqref{p} into \eqref{z4}, we get 
\begin{align}\label{z5}
   \nonumber \int_{B_r\setminus B_{\ov r}}b_m(|x|)|v_m|^{q-2}v_m\phi\ dx&\geq\int_{B_r\setminus B_{\ov r}}b_m(r)|v_m(r)|^{q-2}v_m(r)\phi\ dx\\
   \nonumber &+(q-1)b_m(r)\int_{B_r\setminus B_{\ov r}}\Bigg(\int_{0}^{1}|tv_m(x)+(1-t)v_m(r)|^{q-2}\ dt\Bigg)|\phi|^2\ dx\\
    &\geq \int_{B_r\setminus B_{\ov r}}b_m(r)|v_m(r)|^{q-2}v_m(r)\phi\ dx.
\end{align}
Now, from \eqref{z5}, \eqref{i3}, \eqref{i2} and \eqref{i1}, and that $\overline{u}$ is non-decreasing, it follows from \eqref{z3} that
\begin{align*}
    \int_{B_r\setminus B_{\ov r}}|\nabla\phi|^2\ dx&+\iint_{\R^{2N}\setminus ((B_r\setminus B_{\ov r})^c)^2}\frac{(\phi(x)-\phi(y))^2}{|x-y|^{N+2s}}\ dxdy+\int_{B_r\setminus B_{\ov r}}|\phi|^2\ dx\\
    &\leq \int_{B_r\setminus B_{\ov r}}f_m(\overline{u})\phi\ dx-\int_{B_r\setminus B_{\ov r}}v_m(r)\phi\ dx-\int_{B_r\setminus B_{\ov r}}b_m(r)|v_m(r)|^{q-2}v_m(r)\phi\ dx\\
    &\leq \int_{B_r\setminus B_{\ov r}}f_m(\overline{u}(r))\phi\ dx-\int_{B_r\setminus B_{\ov r}}v_m(r)\phi\ dx-\int_{B_r\setminus B_{\ov r}}b_m(r)|v_m(r)|^{q-2}v_m(r)\phi\ dx\\
    &=\int_{B_r\setminus B_{\ov r}}\Big(f_m(\overline{u}(r))-v_m(r)-b_m(r)|v_m(r)|^{q-2}v_m(r)\Big)\phi\ dx.
\end{align*}
Since by assumption the latter is non-positive, then $\phi\equiv0$, and thus $(i)$ holds.
By the same manner, when $f_m(\overline{u}(r)) > v_m(r)+b_m(r)|v_m(r)|^{q-2}v_m(r)$, we use as a test function
\begin{equation*}
     \phi(x)=\left\{\begin{aligned}
         & 0\quad\quad\quad\quad\quad\quad\quad\quad\quad\quad\text{if}~~\ov r<|x|\leq r\\
         &(v_m(|x|)-v_m(r))^-\quad\quad~\text{otherwise}
     \end{aligned}
     \right.
 \end{equation*}
 to show that case $(ii)$ holds. We then conclude that $v_m$ is non-decreasing and thus $v_m\in K$.

 Next, we show that $\{v_m\}$ is bounded in $V$. Substituting $\phi=v_m$ in \eqref{z2}, we have
 \begin{align}\label{z6}
   \nonumber\int_{B_1}|\nabla v_m|^2\ dx+\iint_{\cQ}\frac{(v_m(x)-v_m(y))^2}{|x-y|^{N+2s}}\ dxdy
     +\int_{B_1}|v_m|^2\ dx&+\int_{B_1}b_m(|x|)|v_m|^{q}\ dx\\
     &=\int_{B_1}f_m(\overline{u})v_m\ dx.
 \end{align}
 Now, using Lemma \ref{llm} together with the fact that $a_m\leq a$, $a\in L^1(0,1)$ and $\overline{u}, v_m\in K$, we have
 \begin{align}
   \nonumber \int_{B_1}f_m(\overline{u})v_m\ dx&=\int_{B_1}a_m(|x|)|\overline{u}|^{p-2}\overline{u}v_m\ dx\leq \int_{B_1}a(|x|)|\overline{u}|^{p-2}\overline{u}v_m\ dx\\
    &\leq\|\overline{u}\|^{p-1}_{L^{\infty}(B_1)}\|a\|_{L^1(B_1)}\|v_m\|_{L^{\infty}(B_1)}\leq C\|v_m\|_{\mathbb{X}},
 \end{align}
 where $C>0$ is a positive constant independent of $m$. Plugging this into \eqref{z6}, it follows that
 \begin{equation}
     \|v_m\|^2_{\mathbb{X}}+\int_{B_1}b_m(|x|)|v_m|^{q}\ dx\leq C\|v_m\|_{\mathbb{X}}.
 \end{equation}
 In particular, $v_m$ is bounded both in $\mathbb{X}$ and $L^q_b(B_1)$. Moreover, by Lemma \ref{llm}, we deduce that $\|v_m\|_{L^{\infty}(B_1)}$ is also bounded. Hence, $\{v_m\}$ is bounded in $V$. Therefore, up to a subsequence, there exists $0\leq v\in \mathbb{X}_{rad}$ such that $v_m\rightharpoonup v$ weakly in $\mathbb{X}$ and $v_m\to v$ strongly in $L^2(B_1)$ (thanks to the compact embedding $\mathbb{X}\hookrightarrow L^2(B_1)$). In particular, $v_m\to v$ a.e. in $B_1$. Since $v_m$ is non-decreasing for every $m$, then $v$ is also non-decreasing in the radial variable. Hence, by Lemma \ref{llm}, $v\in L^{\infty}(B_1)$ from which we obtain in particular that $v\in K$.

 To complete the proof, it remains to show that $v$ solves the problem \eqref{e4}. Let $\phi\in \mathbb{X}\cap L^{\infty}(B_1)$. Then from \eqref{z2}, it holds that
 \begin{align}\label{z7}
   \nonumber &\int_{B_1}\nabla v_m\cdot \nabla\phi\ dx+\iint_{\cQ}\frac{(v_m(x)-v_m(y))(\phi(x)-\phi(y))}{|x-y|^{N+2s}}\ dxdy\\
     &~~~+\int_{B_1}v_m\phi\ dx+\int_{B_1}b_m(|x|)|v_m|^{q-2}v_m\phi\ dx=\int_{B_1}f_m(\overline{u})\phi\ dx.
 \end{align}
 By weak convergence, we have
 \begin{align}\label{z8}
   \nonumber  \int_{B_1}\nabla v_m\cdot \nabla\phi\ dx+\iint_{\cQ}\frac{(v_m(x)-v_m(y))(\phi(x)-\phi(y))}{|x-y|^{N+2s}}\ dxdy+\int_{B_1}v_m\phi\ dx\\
     \to \int_{B_1}\nabla v\cdot \nabla\phi\ dx+\iint_{\cQ}\frac{(v(x)-v(y))(\phi(x)-\phi(y))}{|x-y|^{N+2s}}\ dxdy+\int_{B_1}v\phi\ dx.
 \end{align}
 On the other hand, from the boundedness of $\|v_m\|_{L^{\infty}(B_1)}$, we have
 \begin{equation*}
     |b_m|v_m|^{q-2}v_m\phi|\leq b\|v_m\|^{q-1}_{L^{\infty}(B_1)}\|\phi\|_{L^{\infty}(B_1)}\leq Cb\in L^1.
 \end{equation*}
 Thus, the dominated convergence theorem yields
 \begin{equation}\label{z9}
     \int_{B_1}b_m(|x|)|v_m|^{q-2}v_m\phi\ dx\to \int_{B_1}b(|x|)|v|^{q-2}v\phi\ dx.
 \end{equation}
 Then, passing to the limit in \eqref{z7} and recalling \eqref{z8} and \eqref{z9}, we have that
 \begin{align}\label{z10}
   \nonumber &\int_{B_1}\nabla v\cdot \nabla\phi\ dx+\iint_{\cQ}\frac{(v(x)-v(y))(\phi(x)-\phi(y))}{|x-y|^{N+2s}}\ dxdy\\
     &~~~+\int_{B_1}v\phi\ dx+\int_{B_1}b(|x|)|v|^{q-2}v\phi\ dx=\int_{B_1}a(|x|)|\overline{u}|^{p-2}\overline{u}\phi\ dx~~~\forall\phi\in\mathbb{X}\cap L^{\infty}(B_1).
 \end{align}
 By density, the latter holds for all $\phi\in\mathbb{X}\cap L^1_b(B_1)$. From this, we deduce that $v$ is a weak solution of \eqref{e4}. The proof is therefore finished.
\end{proof}
Having the above preliminary results, we now prove our main results.

\begin{proof}[Proof of Theorem \ref{first-main-result}]
    By Lemma \ref{mpg}, the functional $I_K$ satisfies the \textbf{(PS)} condition. We now prove that $I_K$ satisfies the mountain pass geometry. We first notice that $I_K(0)=0$. Let $e\in K$. We have
    \begin{align*}
        I_K(te)&=\frac{t^2}{2}\int_{B_1}(|\nabla e|^2+|e|^2)\ dx+\frac{t^2}{2}\iint_{\cQ}\frac{(e(x)-e(y))^2}{|x-y|^{N+2s}}\ dxdy\\
        &+\frac{t^q}{q}\int_{B_1}b(|x|)|e|^q\ dx-\frac{t^p}{p}\int_{B_1}a(|x|)|e|^p\ dx.
    \end{align*}
    For $t$ sufficiently large and recalling that $p>q\geq2$, we deduce from the above equation that $I_K(te)$ is negative. Now, let $u\in K$ with $\|u\|_{V}=\rho>0$. From Lemma \ref{llm}, there exists a positive constant $C>0$ such that
    \begin{equation}\label{z11}
        \|u\|_{\mathbb{X}}\leq \|u\|_{V}\leq C\|u\|_{\mathbb{X}}.
    \end{equation}
    Moreover, from the definition of $\|\cdot\|_V$, we have
    \begin{equation}\label{z12}
        \int_{B_1}a(|x|)|u|^p\ dx=\|u\|^p_{L^p_a(B_1)}\leq \|u\|^p_{V}.
    \end{equation}
    Hence, from \eqref{z11} and \eqref{z12}, it follows that
    \begin{align*}
        I_{K}(u)&=\frac{1}{2}\|u\|^2_{\mathbb{X}}+\frac{1}{q}\|u\|^q_{L^q_b(B_1)}-\frac{1}{p}\int_{B_1}a(|x|)|u|^p\ dx\\
        &\geq \frac{1}{2}\|u\|^2_{\mathbb{X}}-\frac{1}{p}\int_{B_1}a(|x|)|u|^p\ dx\\
        &\geq \frac{1}{2C^2}\|u\|^2_{V}-\frac{1}{p}\|u\|^p_{V}=\frac{1}{2C^2}\rho^2-\frac{1}{p}\rho^p>0,
    \end{align*}
    provided that $\rho>0$ is sufficiently small since $p>2$. Now, if $u\notin K$, then we immediately have that $I_K(u)>0$, thanks to the definition of $\Psi_K$, see \eqref{psi-k}. So, $I_K$ satisfies the mountain pass geometry. Hence, by Theorem \ref{mountain-pass-theorem}, the functional $I_K$ admits a nontrivial critical point $\overline{u}\in K$. Now by Lemma \ref{llllm}, there exists $v\in K$ satisfying in the weak sense the problem
    \begin{equation}\label{e4-44}
     \left\{\begin{aligned}
     \cL v+v+b(|x|)|v|^{q-2}v&=a(|x|)|\overline{u}|^{p-2}\overline{u}\quad\quad~\text{in}~~B_1 \\
     \cN_sv&=0 \quad\quad\quad\quad\quad\quad\quad~\text{in}~~\R^N\setminus\overline{B}_1\\
     \frac{\partial v}{\partial\nu}&=0 \quad\quad\quad\quad\quad\quad\quad~\text{on}~~\partial B_1.
     \end{aligned}
     \right.
	\end{equation}
 We now deduce from Theorem \ref{t1} that $\overline{u}$ is a solution of \eqref{e2} with $I_K(\overline{u})>0$. This completes the proof.
\end{proof}
  We shall also address the case $p<q$.  In this case, we need the following assumption:
\begin{itemize}
    \item [(C)] $\Big(\frac{a^q}{b^p}\Big)^{\frac{1}{q-p}}\in L^1(0,1)$ and there exists $e\in K$ such that $I_K(e)<0$.
\end{itemize}
Here is our result for the case $p<q.$

\begin{thm}\label{second-main-result}
    Assume that $(A)$, $(B)$, and $(C)$ hold. If $p<q$ then problem \eqref{e2} admits at least one positive non-decreasing radial solution.
\end{thm}
\begin{proof}[Proof of Theorem \ref{second-main-result}]
    We set $\mu=\inf_{u\in V}I(u)$. From \eqref{k1}, we see that $I_K$ is bounded from below. Moreover, $I_K$ satisfies the \textbf{(PS)} compactness condition by Lemma \ref{mpg}. Thus, it follows from Theorem \ref{y1} that the infimum $\mu$ is achieved at some $\overline{u}\in K$. In particular, $\overline{u}$ is a critical point of $I_K$. It now follows from Lemma \ref{llllm} that there exists $v\in K$, a weak solution of
    \begin{equation}\label{e4-444}
     \left\{\begin{aligned}
     \cL v+v+b(|x|)|v|^{q-2}v&=a(|x|)|\overline{u}|^{p-2}\overline{u}\quad\quad~\text{in}~~B_1 \\
     \cN_sv&=0 \quad\quad\quad\quad\quad\quad\quad~\text{in}~~\R^N\setminus\overline{B}_1\\
     \frac{\partial v}{\partial\nu}&=0 \quad\quad\quad\quad\quad\quad\quad~\text{on}~~\partial B_1.
     \end{aligned}
     \right.
	\end{equation}
 We now deduce from Theorem \ref{t1} that $\overline{u}$ is a solution of \eqref{e2}. Finally, notice that
 \begin{equation*}
     I_K(\overline{u})=\min_{u\in V}I_K(u)\leq I_K(e)<0,
 \end{equation*}
 and thus, $\overline{u}$ is a nontrivial solution.
\end{proof}

\section{Non-constancy of solutions}\label{section:non-constency}
In this section, we show that when $a, b>0$ are two constants and $2\leq q<p$, then the nontrivial solution of 
\begin{equation}\label{e4-444-1}
     \left\{\begin{aligned}
     \cL u+u&=a|u|^{p-2}u-b|u|^{q-2}u\quad\quad~\text{in}~~B_1 \\
     \cN_su&=0 \quad\quad\quad\quad\quad\quad\quad\quad\quad~~~\text{in}~~\R^N\setminus\overline{B}_1\\
     \frac{\partial u}{\partial\nu}&=0 \quad\quad\quad\quad\quad\quad\quad\quad\quad~~~\text{on}~~\partial B_1
     \end{aligned}
     \right.
\end{equation}
obtained in Theorem \ref{first-main-result} is non-constant. We shall see that this is the case under certain conditions on $p$ and $q$. Now we set
\begin{equation*}
    G(u)=a|u|^{p-2}u-b|u|^{q-2}u-u=(a|u|^{p-2}-b|u|^{q-2}-1)u.
\end{equation*}
Then it is observed that every nontrivial constant function $u$ satisfying $G(u)=0$ is also a nontrivial solution. 

In the sequel, we denote by $\lambda_2^{rad}$ the second radial eigenvalue of $\cL+1$ in $B_1$ with Neumann boundary conditions. We have the following.
\begin{lemma}\label{radial-eigenvalue-problem}
    Let $v$ be an eigenfunction associated with $\lambda_2^{rad}$, namely $v$ is nontrivial and satisfies 
    \begin{equation}\label{e4-4444}
     \left\{\begin{aligned}
     \cL v+v&=\lambda_2^{rad}v\quad\quad~\text{in}~~B_1 \\
     \cN_sv&=0 \quad\quad\quad~~~\text{in}~~\R^N\setminus\overline{B}_1\\
     \frac{\partial v}{\partial\nu}&=0 \quad\quad\quad\quad\text{on}~~\partial B_1.
     \end{aligned}
     \right.
\end{equation}
Then $\lambda_2^{rad}>1$, $v$ is radial and unique up to a multiplicative factor. Moreover $\int_{B_1}v\ dx=0$.
\end{lemma}
\begin{proof}
    Let $\lambda_1$ be the first nontrivial Neumann eigenvalue of $\cL+1$ in $B_1$. Then
    \begin{equation}\label{f}
        \lambda_1=1.
    \end{equation}
    Indeed, by definition,
    \begin{align}\label{first-eigenvalue-of-L+1}
        \lambda_1=\inf_{0\neq u\in \mathbb{X}}\frac{\int_{B_1}|\nabla u|^2\ dx+\iint_{\cQ}\frac{(u(x)-u(y))^2}{|x-y|^{N+2s}}\ dxdy+\int_{B_1}|u|^2\ dx}{\int_{B_1}|u|^2\ dx}.
    \end{align}
    Then, in particular, $\lambda_1\geq 1$. Now, using the constant function $u\equiv 1$ as an admissible test function in \eqref{first-eigenvalue-of-L+1} we get that $\lambda\leq 1$ and thus, \eqref{f} follows. Notice that, $\lambda_1=1$ is achieved by the constant function $e_1=1$. 

    Now, let $\lambda_2$ be the second eigenvalue of $\cL+1$ in $B_1$ with Neumann boundary conditions. We claim that
    \begin{equation}\label{lambda-2}
        \lambda_2>1.
    \end{equation}
    To see this, we first notice that $\lambda_2\geq \lambda_1$, that is $\lambda_2=\lambda_1$ or $\lambda_2>\lambda_1$. Let $0\neq w\in\mathbb{X}$ be an eigenfunction associated to $\lambda_2$. Then by definition, 
    \begin{equation}\label{orthoganility-property}
        0=\int_{B_1}we_1\ dx=\int_{B_1}w\ dx.
    \end{equation}
    Now, if $\lambda_2$ were equal to $\lambda_1$ then necessarily $w=e_1=1$ and this violates \eqref{orthoganility-property}. Thus $\lambda_2$ must be strictly greater than $\lambda_1$ and therefore, \eqref{lambda-2} follows. 

    By the inclusion $\mathbb{X}_{rad}\subset\mathbb{X}$ we have  that $\lambda_2^{rad}\geq \lambda_2$ and thus, $\lambda_2^{rad}>1$, thanks to \eqref{lambda-2}. Moreover, by the direct method of calculus of variations, $\lambda_2^{rad}$ is attained at some $v\in \mathbb{X}_{rad}$ with $\int_{B_1}v\ dx=0$.
\end{proof}
Before proving Theorem \ref{non-consistancy-theorem}, let us first notice that under condition \eqref{key-condition-for-non-constancy-of-solutions}, $G$ admits a unique nonzero root. Indeed, set $g(t)=t^{p-2}-bt^{q-2}-1$. Then $g(1)=-b<0$ and $\lim_{t\to\infty}g(t)=+\infty$. Thus from the intermediate value theorem, $g$ has a root in $(1,+\infty)$. Next, $g'(t)=(p-2)t^{p-3}-b(q-2)t^{q-3}=t^{q-3}((p-2)t^{p-q}-b(q-2))>0$ for all $t\in (1, +\infty)$ thanks to \eqref{key-condition-for-non-constancy-of-solutions}. Hence $g$ is strictly increasing in $(1,+\infty)$. We then deduce that $g$ has a unique nonzero root in $(1,+\infty)$. On the other hand, for every $t\in [0,1]$, $g(t)<0$. In conclusion, $G$ has a unique nonzero root. Thus, \eqref{e4-444-1} admits a unique constant solution. Denote it by $u_0\in (1,+\infty)$. 

With the above remark in hand, we now give the proof of Theorem \ref{non-consistancy-theorem}.
\begin{proof}[Proof of Theorem \ref{non-consistancy-theorem}]
    The proof is similar to the one of \cite[Theorem 4.2]{moameni2019existence}. We report it here for the sake of completeness. The goal of the proof is to show that solution $\overline{u}$ of \eqref{e2} obtained in Theorem \ref{first-main-result} is different from $u_0$. Notice that $I(\overline{u})=I_K(\overline{u})=c$ where $c$ is defined (see \eqref{critical-value}) by
\begin{equation}\label{critical-value-2}
       c=\inf_{\gamma\in\Gamma}\sup_{t\in [0,1]}I(\gamma(t)),
\end{equation}
   where $\Gamma=\{\gamma\in C([0,1],V): \gamma(0)=0, \gamma(1)=e\}$.
So, to prove that $\overline{u}\not\equiv u_0$, it suffices to show that $c<I(u_0)$.

Let $v$ be as in Lemma \ref{radial-eigenvalue-problem} and take $\tau>0$ such that $\tau<\frac{\|u_0\|_{L^{\infty}(B_1)}}{\|v\|_{L^{\infty}(B_1)}}$. Then, $u_0+\tau v\in K$. For $r\in\R$, we have
\begin{align*}
    I((u_0+\tau v)r)=\frac{r^2}{2}\int_{B_1}\Big(u_0+\tau\sqrt{\lambda_2^{rad}}v\Big)^2\ dx+\frac{r^q}{q}b\int_{B_1}(u_0+\tau v)^q\ dx-\frac{r^p}{p}\int_{B_1}(u_0+\tau v)^p\ dx.
\end{align*}
Since $q<p$, there is $r>1$ such that $I((u_0+\tau v)r)\leq0$. We introduce the function $\gamma_{\tau}$ defined by $\gamma_{\tau}(t)=t(u_0+\tau v)r$. Notice that $\gamma_{\tau}\in \Gamma$.

Let $\psi:\R^2\to\R$ be given by
\begin{equation*}
    \psi(\tau,t)=I'(t(u_0+\tau v)r)[(u_0+\tau v)r].
\end{equation*}
Since $I$ is of class $C^2$ then $\psi$ is of class $C^1$ with $\psi(0,\frac{1}{r})=0$. Thus,
\begin{align*}
    \frac{d}{dt}\Big|_{(0,\frac{1}{r})}\psi(\tau, t)&=I''(u_0)(ru_0,ru_0)\\
    &=r^2\int_{B_1}\Big(1+b(q-1)u_0^{q-2}-(p-1)u_0^{p-2}\Big) u_0^2\ dx\\
    &=r^2\int_{B_1}u_0^{p}-bu_0^{q}+b(q-1)u_0^{q}-(p-1)u_0^{p}\ dx\\
    &=r^2\int_{B_1}b(q-2)u_0^{q}-(p-2)u_0^{p}\ dx<0.
\end{align*}
In the latter, we have used $G(u_0)=0$ and $b(q-2)<p-2$. Thus, the implicit function theorem guarantees the existence of $\varepsilon_1, \varepsilon_2>0$ and a $C^1$ function $h:(-\varepsilon_1, \varepsilon_1)\to (\frac{1}{r}-\varepsilon_2, \frac{1}{r}+\varepsilon_2)$ such that $h(0)=\frac{1}{r}$ and for $(\tau, t)\in (-\varepsilon_1, \varepsilon_1)\times (\frac{1}{r}-\varepsilon_2, \frac{1}{r}+\varepsilon_2)$ one has $\psi(\tau, t)=0$ if and only if $t=h(\tau)$. On the other hand,
\begin{align*}
    \frac{d}{d\tau}\Big|_{(0,\frac{1}{r})}\psi(\tau,t)&=I''(u_0)(u_0,rt)+I'(u_0)rv\\
    &=r\int_{B_1}\Big(1+b(q-1)u_0^{q-2}-(p-1)u_0^{p-2}\Big)u_0v\ dx\\
    &=r\Big(1+b(q-1)u_0^{q-2}-(p-1)u_0^{p-2}\Big)u_0\int_{B_1}v\ dx=0.
\end{align*}
Note that we have used the fact that $I'(u_0)=0$ and $\int_{B_1}v\ dx=0$. Thus, $h'(0)=0$.

Next, we claim that
\begin{equation}\label{lab}
    I(h(\tau)(u_0+\tau v)r)< I(u_0),\quad\quad\text{for all}~~\tau\in (-\varepsilon_1, \varepsilon_1).
\end{equation}
To see this, we first notice that since $h'(0)=0$, then for all $\tau\in (-\varepsilon_1, \varepsilon_1)$, we have $h(\tau)=\frac{1}{r}+o(\tau)$ so that
\begin{equation*}
    h(\tau)(u_0+\tau v)r-u_0=(rh(\tau)-1)u_0+\tau rh(\tau)v=\tau v+o(\tau).
\end{equation*}
Now recalling that $I'(u_0)=0$, Taylor expansion yields
\begin{align*}
    &I(h(\tau)(u_0+\tau v)r)-I(u_0)\\
    &=\frac{1}{2}I''(u_0)[h(\tau)(u_0+\tau v)r-u_0, h(\tau)(u_0+\tau v)r-u_0]+o(\tau^2)\\
    &=\frac{1}{2}I''(u_0)(\tau v+o(\tau), \tau v+o(\tau))+o(\tau^2)=\frac{\tau^2}{2}I''(u_0)(v,v)+o(\tau^2)\\
    &=\frac{\tau^2}{2}\Bigg(\int_{B_1}|\nabla v|^2\ dx+\iint_{\cQ}\frac{(v(x)-v(y))^2}{|x-y|^{N+2s}}\ dxdy+\int_{B_1}|v|^2\ dx\Bigg)\\
    &~~~~~~~+\frac{\tau^2}{2}\Bigg(\int_{B_1}\Big(b(q-1)u_0^{q-2}v^2-(p-1)u_0^{p-2}v^2\Big)\ dx\Bigg)+o(\tau^2)\\
    &=\frac{\tau^2}{2}\int_{B_1}\Big(\lambda_2^{rad}v^2+b(q-1)u_0^{q-2}v^2-(p-1)u_0^{p-2}v^2\Big)\ dx+o(\tau^2)\\
    &=\frac{\tau^2}{2u_0^2}\int_{B_1}\Big(\lambda_2^{rad}u_0^2+b(q-1)u_0^q-(p-1)u_0^p\Big)v^2\ dx+o(\tau^2)\\
    &=\frac{\tau^2}{2u_0^2}\int_{B_1}\Big(\lambda_2^{rad}u_0^p-\lambda_2^{rad}bu_0^q+b(q-1)u_0^q-(p-1)u_0^p\Big)v^2\ dx+o(\tau^2).
\end{align*}
In the latter, we have used that $G(u_0)=0$. Hence,
\begin{align}\label{lab1}
    I(h(\tau)(u_0+\tau v)r)-I(u_0)&=\frac{\tau^2}{2u_0^2}\int_{B_1}\Big((b(q-1)-\lambda_2^{rad})u_0^q-((p-1)-\lambda_2^{rad})u_0^p\Big)v^2\ dx+o(\tau^2).
\end{align}
Now, using \eqref{key-condition-for-non-constancy-of-solutions} and the fact that $\lambda_2^{rad}>1$ (see Lemma \ref{radial-eigenvalue-problem}), we have
\begin{equation*}
    b\leq \frac{p-2}{q-2}<\frac{p-1-\lambda_2^{rad}}{q-1-\lambda_2^{rad}}.
\end{equation*}
Taking this into account, the right-hand side in \eqref{lab1} is strictly negative, and thus \eqref{lab} follows.

On the other hand the function $t\mapsto I(\gamma_0(t))$ takes its unique maximum at $t=\frac{1}{r}$. Indeed,
\begin{align}
    \frac{d}{dt}I(\gamma_0(t))&=\frac{d}{dt}I(tu_0r)=I'(tu_0r)ru_0\\
    &=|B_1|\Big(tu_0^2r^2+bt^{q-1}u_0^qr^q-t^{p-1}u_0^pr^p\Big)u_0\\
    &=-|B_1|tu_0^2r^2g(tu_0r),
\end{align}
where $|B_1|$ is the measure of $B_1$ and $g(t)=t^{p-2}-bt^{q-2}-1$. Thus, $\frac{d}{dt}I(\gamma_0(t))>0$ if $t<\frac{1}{r}$ and $\frac{d}{dt}I(\gamma_0(t))<0$ if $t>\frac{1}{r}$. It therefore follows that $t\mapsto I(\gamma_0(t))$ has a unique maximum point at $t=\frac{1}{r}$.

Since $(\tau, t)\mapsto I(\gamma_{\tau}(t))$ is a continuous function, we can choose $0<\tau_0<\varepsilon_1$ sufficiently small such that the maximum of the function $t\mapsto I(\gamma_{\tau_0}(t))$ lies in $(\frac{1}{r}-\varepsilon_2, \frac{1}{r}+\varepsilon_2)$. Let $t_0\in (\frac{1}{r}-\varepsilon_2, \frac{1}{r}+\varepsilon_2)$ be such a maximum point. Then
\begin{equation*}
    0=\frac{d}{dt}I(\gamma_{\tau_0}(t))\Big|_{t=t_0}=I'(t_0(u_0+\tau_0v)r)(u_0+\tau_0v)r,
\end{equation*}
and therefore $t_0=h(\tau_0)$. From \eqref{lab} it follows that
\begin{equation*}
    I(t_0(u_0+\tau_0 v)r)<I(u_0).
\end{equation*}
Hence,
\begin{equation*}
     I(t(u_0+\tau_0 v)r)<I(u_0),\quad\quad \forall t\in [0,1].
\end{equation*}
Thus, $\max_{t\in [0,1]}I(\gamma_{\tau_0}(t))< I(u_0)$. This yields
\begin{equation*}
    c\leq \max_{t\in [0,1]}I(\gamma_{\tau_0}(t))< I(u_0),
\end{equation*}
as desired. The proof is finished.
\end{proof}

\section*{Data availability statement}
Data sharing not applicable to this article as no datasets were generated or analyzed during the current study.

\section*{Declaration of competing interest}

The authors declare that they have no known competing financial interests or personal relationships that could have
appeared to influence the work reported in this paper.

\section*{Authors contributions} 

All authors contributed equally. \\

\textbf{Acknowledgements:} D.A. and A.M. are pleased to acknowledge the support of the Natural Sciences and Engineering Research Council of Canada. R.Y.T. is supported by Fields Institute. The authors are grateful to referees for his/her valuable comments for improvement of this article.

\bibliographystyle{ieeetr}

\end{document}